\documentclass[a4paper]{amsart}
\usepackage{math}
\usepackage{lineno,mathabx,mathrsfs,mathtools,amsrefs}

\newtheorem*{mainthm}{Theorem}

\newtheorem*{mainproblem}{Problem}

\renewcommand\End{\operatorname{End}}
\newcommand\Aff{\operatorname{Aff}}

\begin{document}
\title{On Gardam's and Murray's units in group rings}
\author{Laurent Bartholdi}
\address{Fachrichtung Mathematik, Universität des Saarlandes}
\email{laurent.bartholdi@gmail.com}
\date{December 15th, 2022}
\begin{abstract}
  We show that the units found in torsion-free group rings by Gardam are twisted unitary elements. This justifies some choices in Gardam's construction that might have appeared arbitrary, and yields more examples of units. We note that all units found up to date exhibit non-trivial symmetry.
\end{abstract}
\maketitle


\section{Introduction}
In~\cite{gardam:units}, Giles Gardam came up with a remarkable unit in a torsion-free group ring, providing a counterexample to a (strong form of a) 80-year old conjecture by Kaplansky~\cite{kaplansky:problemrings,kaplansky:problemrings2}. This was later generalized by Alan Murray~\cite{murray:units} to a family of counterexamples in arbitrary non-zero characteristic.

Even though the problem in characteristic $0$ remains open, these results motivated a variety of researchers in understanding the properties of these units and their possible extensions and generalizations.\footnote{Gardam's proof has even been formalized in Lean:\\
  {\scriptsize\url{https://github.com/todbeibrot/counter-example-to-the-unit-conjecture-for-group-rings}},\\
    {\scriptsize\url{https://github.com/siddhartha-gadgil/Polylean/blob/unit-conjecture}}.}

We concentrate on a specific group $P$, the ``Hantsche-Wendt group'', which is the one in which non-trivial units were found by Gardam and followers. Our main results are:
\begin{itemize}
\item Up to applying endomorphisms and translations, all known units are \emph{unitary} and \emph{symmetric};
\item In odd characteristic a single orbit of units is known;
\item In characteristic $2$ there are at least two orbits of units, one generated by Gardam's unit and one generated by a new unit.
\end{itemize}

\subsection{Acknowledgments}
I am very grateful to Giles Gardam for generous suggestions and feedback.

\section{The group and its group ring}
The group considered in all examples is the ``Hantzsche-Wendt manifold group'', the only three-dimensional Bieberbach group\footnote{Namely, torsion-free crystallographic; that is, fundamental group of a flat manifold.} with trivial centre and point group $C_2\times C_2$, and therefore the first potentially interesting group for this problem. The group is
\[P=\langle a,b\mid a^{2b}=a^{-2},b^{2a}=b^{-2}\rangle,\]
with index-$4$ abelian subgroup $\langle x=a^2,y=b^2,z=(ab)^2\rangle\cong\Z^3$.

Let $\Bbbk$ be a commutative ring, and consider the group ring $\Bbbk P$. Recall that $\Bbbk P$ is a $*$-algebra for the anti-involution
\[u=\sum_{g\in P} n_g g\mapsto u^*\coloneqq\sum_{g\in P} n_g g^{-1}.\]
We shall consider a version of this anti-involution, twisted by automorphisms. The outer automorphism group of $P$ was computed in~\cite{zimmermann:hw}; it consists of $96$ elements, with remarkably few of these lifting to finite-order automorphisms of $P$. There are, nevertheless, two order-$2$ automorphisms $\sigma,\tau$ of $P$ that shall play an important role for us, given respectively by
\[a^\sigma=a,\quad b^\sigma=b^{-1},\qquad a^\tau=a^{-1},\quad b^\tau=x^2b^{-1}.\]
Consider also the multiplicative character $\chi\colon P\to\{\pm1\}$ given by
\[a^\chi=-1,\quad b^\chi=-1.\]
Together, $\sigma$ and $\chi$ induce an automorphism $\theta$ of the group ring $\Bbbk P$ by
\[u\sum_{g\in P} n_g g\mapsto u^\theta\coloneqq\sum_{g\in P}g^\chi n_g g^\sigma.\]
Similarly, $\tau$ induces an automorphism of $\Bbbk P$.
(Other choices of automorphisms are possible, but these are those that are closest to Murray's example, as we shall see below.)

Note that $\theta$ commutes with $*$ and $\theta^2=1$, so $u\mapsto u^{*\theta}$ is again an anti-involution. By definition, an element $u\in \Bbbk P$ shall be called \emph{$\theta$-unitary} if it satisfies $u^{*\theta}u=1$. Note that every $g\in\operatorname{Fix}(\sigma)\cap\ker(\chi)$ is $\theta$-unitary; these are \emph{trivial} unitaries. Our main remark is:
\begin{mainthm}
  If $\Bbbk$ is a field of non-zero characteristic, then there are non-trivial $\theta$-unitary units in $\Bbbk P$. Furthermore, these units may be chosen to be invariant under $\tau$.
\end{mainthm}
\begin{proof}
  Let $p\ge2$ be the characteristic of $\Bbbk$. We consider the unit given by Murray~\cite{murray:units}*{Theorem~3}, making the choice $t=0,w=1$, left-multiplying by $a z$, and right-multiplying by $b^{-1}$: it is
  \begin{multline*}
    u = 1 + (1-z^{-1})^{p-2}\Big((4+x+x^{-1}+y+y^{-1}) + \big((1+y^{-1})(x^{-1}+y)+(1+x^{-1})(1+z^{-1})\big) a\\
    + \big((1+x^{-1})(x+y^{-1})+(1+y^{-1})(1+z)\big)b
    + (1+x^{-1})(1+y)z^{-1}(1+z^{-1})a b\Big);
  \end{multline*}
  and a tedious but straightforward calculation shows $u^{*\theta}u=1$ and $u^\tau=u$. In more details: write $u=1+p+qa+rb+sab$; and for a polynomial $p(x,y,z)$ write $p_x=p(x^{-1},y,z)$ and similarly for the other variables. Then
  \begin{align*}
    u^* &= p_{xyz} + a^{-1}q_{xyz} + b^{-1}r_{xyz} + (ab)^{-1}s_{xyz}\\
        &= p_{xyz} + x^{-1}q_x a + y^{-1}r_y b + z^{-1}s_z ab,\\
    u^\theta &= p_y + q_y a - y^{-1}r_y b - y s_y ab,\\
    u^\tau &= p_{xy} + q_{xy}x^{-1}a + r_{xy}y^{-1}b + s_{xy}x^{-1}y ab.
  \end{align*}
  These calculations are really a reformulation of the seemingly-arbitrary choices made by Gardam.\footnote{He writes $u=p+q a+r b+s ab$ for Laurent polynomials $p(x,y,z),\dots$ and posits an inverse of the form $u'=x^{-1}p(x,y^{-1},z^{-1})-x^{-1}q a-y^{-1}r b-z^{-1}s(x,y^{-1},z^{-1})a b$. His calculations are, after the appropriate left- and right-multiplications, precisely the ones that show that $u$ is $\theta$-unitary. His special choices of $p,q,r,s$ also amount to requiring $u$ to be $\tau$-invariant.}
\end{proof}

\subsection{The unitary subgroup}
We may naturally ask ``how many non-trivial units are there in $\Bbbk P$?''. To make sense of this question, we consider the ``endomorph'' $\Aff(P)\coloneqq P\rtimes\End(P)$ consisting of maps $g\mapsto h\cdot g^\alpha$ for some $h\in P,\alpha\in\End(P)$. The automorphism group of $P$ is mentioned above, and is a finite extension of $P$; there are furthermore endomorphisms $a\mapsto a^i,b\mapsto b^j$ for any odd $i,j$. Together with $P$, they generate the affine monoid of $P$. Furthermore, the affine group of $\Bbbk$, consisting of maps $x\mapsto y\cdot x^\beta$ for some $y\in\Bbbk^\times$ and some field endomorphism $\beta$ of $\Bbbk$, also act on the units of $\Bbbk P$.

Thus a natural question is how many $\Aff(P)\times\Aff(\Bbbk)$-orbits are needed to generate $(\Bbbk P)^\times$.

It would be strange that $(\Bbbk P)^\times$ be generated by a single orbit, but I cannot rule this out.\footnote{A generator could be the element $u$ above; note that trivial units are obtained as the image of $u$ by constant endomorphs. It was noted by Passman~\cite{passman:units}*{Proposition~2} that the freedom afforded by the parameters $t,w$ in Murray's examples amount to applying a group automorphism (to wit, $\{a\mapsto z^{w-t}a,b\mapsto z^w b\}$).} This is not the case in characteristic $2$, where there are at least two orbits.

Indeed, over $\mathbb F_2$, I have found, by computer search, two other units, supported respectively on $57$ and $67$ elements; let us call them $u_{57}$ and $u_{67}$, reserving the name $u_{21}$ for Gardam's unit. They are also  $\theta$-unitary and $\tau$-symmetric. Here they are:
\begin{align*}
  u_{57} &= \scriptscriptstyle\big(xyz^{-1}+x^{-1}z^{-1}+xy^{-1}z^{-1}+y^{-2}+1+x^{-2}+yz^{-1}+y^{2}+y^{-1}z^{-1}+x^{-1}yz^{-1}+x^{-1}y^{-1}z^{-1}+xyz+xy^{-1}z+x^{2}+xz^{-1}+x^{-1}yz+x^{-1}y^{-1}z\big)\\
         &+ \scriptscriptstyle\big(y^{-1}z^{2}+xz+x^{-1}y^{-1}z^{2}+yz^{2}+x^{-1}yz^{2}+xy+y^{2}z+y^{-1}+x^{-1}y+x^{-1}y^{-2}z+yz+x^{-1}y^{-1}z+x^{-2}y^{-1}+1+x^{-2}z+x^{-1}\big)a\\
         &+ \scriptscriptstyle\big(y^{-1}z+x^{-1}y^{-1}+xy^{-1}z+x^{-2}y^{-1}+x^{2}+x^{-1}z+z+x^{-1}y^{-1}z^{-1}+x+xyz+xy^{-1}z^{-1}+x^{-1}z^{-1}+y+x^{-1}y^{-2}z+xz^{-1}+y^{-2}\big)b\\
         &+ \scriptscriptstyle\big(x^{-1}z^{-2}+yz^{-2}+z^{-2}+x^{-1}yz^{-3}+x^{-1}z^{-3}+yz^{-3}+z^{-3}+x^{-1}yz^{-2}\big)ab,\\
  u_{67} &= \scriptscriptstyle\big(xyz^{-1}+xy^{-1}z^{-1}+y^{-2}+1+xy+xz+xy^{-1}+y^{2}+y^{-1}z^{-1}+yz^{-1}+x^{-1}yz^{-1}+x^{-1}y^{-1}z^{-1}+xyz+x^{-1}z+x^{-1}y+xy^{-1}z+x^{-1}y^{-1}+x^{-1}yz+x^{-1}y^{-1}z\big)\\
         &+ \scriptscriptstyle\big(y^{-2}+xy^{-1}+xz+y^{-1}+z+y+x^{-1}y^{-1}+x^{-1}z+x^{-2}yz+x^{-1}y+xy^{-1}z+z^{-1}+x^{-1}z^{-1}+yz+x^{-1}y^{-1}z+x^{-2}+y^{-1}z^{-1}+x^{-1}y^{-1}z^{-1}+yz^{-1}+x+x^{-1}yz^{-1}+x^{-2}z+x^{-1}y^{2}+x^{-2}y\big)a\\
         &+ \scriptscriptstyle\big(xy^{-1}z+x^{-1}z+xz+xy^{-2}z^{-1}+y+x^{-1}yz^{-1}+x^{-1}y^{-1}z+y^{-2}\big)b\\
         &+ \scriptscriptstyle\big(z^{-1}+xz^{-1}+xyz^{-2}+x^{-2}yz^{-1}+xz^{-2}+x^{-1}yz^{-3}+x^{-2}z^{-1}+x^{-1}yz^{-1}+x^{-2}yz^{-2}+x^{-1}z^{-3}+yz^{-3}+x^{-1}z^{-1}+yz^{-1}+x^{-2}z^{-2}+z^{-3}+xyz^{-1}\big)ab.
\end{align*}

There is a natural map $\pi\colon P\to D_\infty=\langle a,b\mid b^2,a^b=a^{-1}\rangle$, given by $a\mapsto a,b\mapsto b$. Note that $\theta$ induces the identity map on $\mathbb F_2[D_\infty]$, and that $\tau$ induce the map $a\mapsto a^{-1},b\mapsto b$. Thus every unit $u$ of $\mathbb F_2[P]$ that is $\theta$-unitary or $\tau$-symmetric induces a unit $u^\pi$ of $\mathbb F_2[D_\infty]$ that is unitary.

The units of $\mathbb F_2[D_\infty]$ were computed by Mirowicz~\cite{mirowicz:units}; let us recall the result, restricting to unitary units. They are precisely of the form
\[\varepsilon_{i,j}=b^j + (a^{-i} + a^i)(1+b),\]
for $i\in\N$ and $j\in\{0,1\}$, and generate an elementary abelian group of countably infinite rank. Now a direct calculation gives
\begin{align*}
  u_{21}^\pi &= \varepsilon_{2,0},\\
  u_{57}^\pi &= \varepsilon_{0,0}+\varepsilon_{2,0}+\varepsilon_{4,0},\\
  u_{67}^\pi &= \varepsilon_{0,0}+\varepsilon_{2,0}+\varepsilon_{4,0}.
\end{align*}
Now the action of the affine semigroup of $P$ is as follows: automorphisms and translations (when they preserve $\theta$-unitarity and $\tau$-symmetry) preserve $\varepsilon_{i,j}$; while the endomorphism $a\mapsto a^{2k+1},b\mapsto b^{2\ell+1}$ maps $\varepsilon_{i,j}$ to $\varepsilon_{(2k+1)i,j}$. Therefore, the orbit of $u_{21}$ maps in the subgroup $\langle 1,b,\varepsilon_{4i+2,j}\rangle$ of $\mathbb F_2[D_\infty]$ so it does not contain $u_{57}$. All in all, the $\theta$-unitary, $\tau$-symmetric units of $\mathbb F_2[P]$ map under $\pi$ to a subgroup of $\mathbb F_2[D_\infty]^\times$ containing at least $\langle\varepsilon_{2i,j}:i\in\N,j\in\{0,1\}\rangle$.

\subsection{A matrix representation}
Since $P$ admits an index-$4$ abelian subgroup, $\Bbbk P$ admits a faithful representation $\rho$ by $4\times4$-matrices over $\Bbbk[x^{\pm1},y^{\pm1},z^{\pm1}]$ given by
\[a\mapsto\begin{pmatrix}
            0 & 1 & 0 & 0\\
            x & 0 & 0 & 0\\
            0 & 0 & 0 & x^{-1}z^{-1}\\
            0 & 0 & y^{-1}z & 0
          \end{pmatrix},\quad
  b\mapsto\begin{pmatrix}
            0 & 0 & 1 & 0\\
            0 & 0 & 0 & 1\\
            y & 0 & 0 & 0\\
            0 & y^{-1} & 0 & 0
          \end{pmatrix}.
\]
Craven and Pappas noted in~\cite{craven-pappas:unit}*{Theorem~6.8} that every unit of $\Bbbk P$ has determinant in $\Bbbk^\times$, essentially because $P$ has trivial centre. It is immediate that $(u^*)^\rho=(u^\rho)^*$ for the anti-involution $*$ on matrices given by transposing and inverting $x,y,z$; and that the involution $\theta$ may be realized in the matrix representation as conjugation by the diagonal matrix $[1;1;-y^{-1};-y]$ followed by inverting $y$. Thus $u^\rho$ is also twisted-unitary, and this gives an alternate reason for its determinant to be in $\Bbbk^\times$.

\section{Final remarks}

Gardam arrived at his example by performing an extensive computer search. Selecting small subsets $S,T$ of $P$, one considers variables $(u_s)_{s\in S},(v_t)_{t\in T}$, their products $w_{s,t}=u_s v_t$, and solves $(\sum_{s t=g}w_{s,t}=[g=1])_{g\in ST}$. This is particularly efficiently implementable if $\Bbbk=\mathbb F_2$, in which case the $u_s,v_t,w_{s,t}$ are boolean variables with $w_{s,t}=u_s\wedge v_t$; this system can then be fed to a SAT solver.

This system is barely solvable with currently-available (laptop) solvers; for example, selecting for $S=T=B_4$ the ball of radius $4$ in $\{a,b,ab\}^{\pm1}$ does not find a solution in reasonable time. However, searching for a $\theta$-unitary example \emph{does} provide a solution over $\mathbb F_2$ without needing to guess a specific form of the supports $S,T=S^{-\theta}$: it suceeds in a few minutes with $S=B_4$.

There is a moral behind this: searching for a twisted-unitary unit halves the number of unknowns (so in exponential-time-search square-roots the running time, whatever that means); but if $u$ is $\theta$-unitary then $w=u^{*\theta}u=1$ is automatically $\theta$-self-ajoint ($w^{*\theta}=w$) so the number of constraints is also halved, and the likelihood of finding a solution (again, whatever that means) does not drop significantly.

There are further alternative ways of trimming down the search space, and requiring a solution to be $\tau$-invariant is one. I would have preferred to find even more invariant solutions, for example under the order-$6$ automorphism $\{a\mapsto b,b\mapsto a b\}$, but did not succeed, or at least its square, but did not succeed. I have also not been able to find unitary solutions, or twisted-unitary solutions for other automorphisms than $\theta$.

We finally note that the unit conjecture for $\Z P$ remains open. However,
\begin{prop}
  For every $n\in\N$ there exist elements $u,u'\in\Z P$ with
  \[u' u\equiv1\pmod{n\Z P}.\]
\end{prop}
\begin{proof}
  The solutions modulo $p$ afforded by the theorem (or the previous constructions) may be lifted to solutions modulo any product of distinct primes, by the Chinese Remainder Theorem. On the other hand, if $p$ divides $n$ then every solution modulo $n$ may be lifted to a solution modulo $p n$: given $u'_0 u_0\equiv 1\pmod n$, write $u'_0 u_0=1+nw$, and then write $u=u_0+n v$, $u'=u'_0+n v'$ for unknowns $v,v'$. Consider the resulting equation
  \[u'u\equiv u'_0u_0 + n(v'u_0+u'_0v)= 1+n(w+v'u_0+u'_0v)\pmod{p n},\]
  namely $w+v'u_0+u'_0v\equiv0\pmod p$. If the unknowns $v,v'$ are supported on a ball of radius $r$ in $P$, and $u_0,u'_0,w$ are supported on a ball of radius $q$, then this equation takes place on the ball of radius $q+r$ in $P$. It has thus $\#B_{q+r}$ constraints and $2\#B_r$ variables. Since $\#B_r$ grows as a polynomial of degree $3$ in $r$, there will exist a solution as soon as $r$ is large enough.
\end{proof}
A counterexample to the unit conjecture over $\Z$ would follow if the solutions $v,v'$ could be forced to be supported on a fixed finite subset of $P$. Indeed repeating the process would lead to a $p$-adic solution, which can then easily be brought to an algebraic, and then integer, form.

This argument naturally breaks down if the solution is furthermore required to be twisted-unitary; one obtains $w^{*\theta}=w$, and seeks $v$ with $w+v^{*\theta}u_0+u_0^{*\theta}v\equiv0\pmod p$; this system is typically overconstrained, and I have been unable to find any solution. I leave as a challenge:
\begin{mainproblem}
  Is there a non-trivial\footnote{By ``non-trivial'' one should ask for a solution that does not reduce to a trivial one modulo a prime; thus $1+2b$ is a trivial solution modulo $4$.} twisted-unitary solution $u^{*\theta}u\equiv1\pmod{4\Z P}$?
\end{mainproblem}


\end{document}